\documentclass[12pt]{article}
\usepackage{amsmath,amsthm,amssymb}
\usepackage{mathrsfs}
\DeclareMathOperator{\tr}{tr}

\begin{document}
\input amssym.def
\setcounter{equation}{0}
\newcommand{\wt}{{\rm wt}}
\newcommand{\spa}{\mbox{span}}
\newcommand{\Res}{\mbox{Res}}
\newcommand{\End}{\mbox{End}}
\newcommand{\Ind}{\mbox{Ind}}
\newcommand{\Hom}{\mbox{Hom}}
\newcommand{\Mod}{\mbox{Mod}}
\newcommand{\m}{\mbox{mod}\ }
\renewcommand{\theequation}{\thesection.\arabic{equation}}
\numberwithin{equation}{section}

\def \End{{\rm End}}
\def \Aut{{\rm Aut}}
\def \Z{\mathbb Z}
\def \H{\mathbb H}
\def \MM{\Bbb M}
\def \C{\mathbb C}
\def \R{\mathbb R}
\def \Q{\mathbb Q}
\def \N{\mathbb N}
\def \ann{{\rm Ann}}
\def \<{\langle}
\def \o{\omega}
\def \O{\Omega}
\def \Or{\cal O}
\def \M{{\cal M}}
\def \1t{\frac{1}{T}}
\def \>{\rangle}
\def \t{\tau }
\def\W{\cal W}
\def \a{\alpha }
\def \e{\epsilon }
\def \l{\lambda }
\def \L{\Lambda }
\def \g{\gamma}
\def \b{\beta }
\def \om{\omega }
\def \o{\omega }
\def \ot{\otimes}
\def \cg{\chi_g}
\def \ag{\alpha_g}
\def \ah{\alpha_h}
\def \ph{\psi_h}
\def \S{\cal S}
\def \nor{\vartriangleleft}
\def \V{V^{\natural}}
\def \voa{vertex operator algebra\ }
\def \voas{vertex operator algebras}
\def \v{vertex operator algebra\ }
\def \1{{\bf 1}}
\def \be{\begin{equation}\label}
\def \ee{\end{equation}}
\def \qed{\mbox{ $\square$}}
\def \pf {\noindent {\bf Proof:} \,}
\def \bl{\begin{lem}\label}
\def \el{\end{lem}}
\def \ba{\begin{array}}
\def \ea{\end{array}}
\def \bt{\begin{thm}\label}
\def \et{\end{thm}}
\def \br{\begin{rem}\label}
\def \er{\end{rem}}
\def \ed{\end{de}}
\def \bp{\begin{prop}\label}
\def \ep{\end{prop}}
\def \p{\phi}
\def \d{\delta}
\def \irr{\rm irr}

\newtheorem{th1}{Theorem}
\newtheorem{ree}[th1]{Remark}
\newtheorem{thm}{Theorem}[section]
\newtheorem{prop}[thm]{Proposition}
\newtheorem{coro}[thm]{Corollary}
\newtheorem{lem}[thm]{Lemma}
\newtheorem{rem}[thm]{Remark}
\newtheorem{de}[thm]{Definition}
\newtheorem{hy}[thm]{Hypothesis}
\newtheorem{conj}[thm]{Conjecture}
\newtheorem{ex}[thm]{Example}

\begin{center}
{\Large {\bf Congruence Property in Orbifold Theory}}\\
\vspace{0.5cm}

Chongying Dong\footnote
{Supported by a NSF grant DMS-1404741 and China NSF grant 11371261}
\\
 Department of Mathematics, University of
California, Santa Cruz, CA 95064 USA \\
Li Ren\footnote{Supported by China NSF grant 11301356}\\
 School of Mathematics,  Sichuan University,
Chengdu 610064 China
\end{center}

\begin{abstract} Let $V$ be a rational, selfdual,  $C_2$-cofinite vertex operator algebra of CFT type,  and $G$ a finite automorphism group of $V.$  It is proved  that the kernel of the representation of the modular group on  twisted conformal blocks associated to $V$ and $G$  is a congruence subgroup. In particular, the $q$-character of each irreducible twisted module is a modular function on the same congruence subgroup. In the case $V$ is the Frenkel-Lepowsky-Meurman's moonshine vertex operator algebra 
and $G$ is the monster simple group, the generalized McKay-Thompson series associated to any commuting pair in the monster group is a modular function.
\end{abstract}

\section{Introduction}

The modular invariance of trace functions in orbifold theory was established in \cite{DLM4}.  That is, the twisted conformal blocks
associated to a rational, $C_2$-cofinite vertex operator algebra $V$ and a finite automorphism group $G$  afford  a representation of the modular group $SL(2,\Z).$ We prove in this paper that the kernel of this representation is a congruence subgroup of
$SL(2,\Z).$

The modular invariance of trace functions associated to any rational and $C_2$-cofinite vertex operator algebra was established 
in \cite{Z} although the modular invariance of the characters of the  rational conformal field theory  \cite{C}  in physics,  
and the characters  of the  minimal representations of the Virasoro algebra \cite{R}  and  integrable highest weight representations for the affine Kac-Moody algebras \cite{KP}  in mathematics was known much earlier.  It was conjectured by many people that the kernel  of  the representation of $SL(2,\Z)$  in \cite{Z} is a congruence subgroup. This conjecture was  recently settled down in \cite{DLN}. 

According to a well known conjecture in orbifold theory, $V^G$ is rational for any rational vertex operator algebra $V$ and any finite
automorphism group $G.$ It is easy to show that the trace functions defined in \cite{DLM4} are vectors in the twisted  conformal blocks 
associated vertex operator algebra $V^G.$ So it is natural to expect that the kernel of the representation of $SL(2,\Z)$ on the twisted  conformal blocks associated to $V$ and $G$ is a congruence subgroup. Unfortunately, the rationality and the $C_2$-cofiniteness of $V^G$ were proved only when $G$ is solvable in \cite{M}, \cite{CM}.  But this result for solvable group is sufficient for the congruence subgroup property in general orbifold theory.   The key observation  is that a trace function in the orbifold theory
only involves with a commuting pairing $(g,h)$ in $G.$ Replacing $G$ by the subgroup $H$ generated by $g,h,$ we see
that $V^H$ is rational and $C_2$-cofinite. We can then apply the congruence subgroup property result in \cite{DLN} to prove 
any trace function in the twisted conformal blocks of  orbifold theory is modular on a congruence subgroup. In particular, the $q$-character $\chi_{V^G}(\tau)$ of $V^G$ is a modular function. This leads us to a conjecture which characterizes the rational  vertex operator algebra by the modularity of  its $q$-character.  A successful proof of this conjecture would imply that $V^G$ is rational
for any finite automorphism group $G.$

One important application is the modularity of the generalized McKay-Thompson series $Z(g,h,\tau)$ associated to the moonshine vertex operator algebra $V^{\natural}$  and its full automorphism group - the monster simple group $\MM$ \cite{FLM}, \cite{G}. 
Norton's generalized moonshine conjecture \cite{N}  says that for any commuting paring $(g,h)$ in $\MM$ there is a genus zero modular
function $Z(g,h,\tau)$ such that (1) The coefficient of  each power of  $q$ in  $Z(g,h,\tau)$ is a projective character of 
$C_{\MM}(g),$ (2) $Z(g,h,\tau)$ is invariant under the conjugation, (3) For any $ \gamma=\left(\begin{array}{cc}a & b\\ c & d\end{array}\right)$ in $SL(2,\Z),$ there are nonzero constants $\gamma_{(g,h), (g^ah^c,g^bh^d)}$ such that 
$$Z(g,h,\gamma\tau)= \gamma_{(g,h),(g^ah^c,g^bh^d)} Z(g^ah^c, g^bh^d, \tau),$$
(4) $Z(1,g,\tau)$ is the original Mckay-Thompson series \cite{CN}. 
It was pointed out in \cite{DLM4} that the trace   functions $Z(g,h,\tau)$ appearing in the orbifold theory associated to $V^{\natural}$ and $\MM$ satisfy conditions (1)-(4).  As a corollary of the main result in this paper, we see that each  $Z(g,h,\tau)$
is a modular function over a congruence subgroup. But the genus zero property of   $Z(g,h,\tau)$  remains open except that  the subgroup generated by $g$ and $h$ is  cyclic \cite{B}, \cite{DLM4}.

 The paper is organized as follows.  Section 2 covers the twisted modules, $g$-rationality  and related concepts following \cite{DLM2}, \cite{DLM3}. We also present  some  important  results  for $g$-rational vertex operator algebras. Section 3  is a review 
of the modular invariance of trace functions in orbifold theory from \cite{DLM4}.  We prove the main result in Section 4. That is,
the kernel of the representation of $SL(2,\Z)$ on the twisted  conformal blocks in orbifold theory is a congruence subgroup. 
In the last section we consider the special case when the vertex operator algebra is holomorphic and we also discuss how these 
results are related to the generalized moonshine conjecture \cite{N}  for the moonshine vertex operator algebra $V^{\natural}$ and
the monster group $\MM.$

\section{Basics}

In this section we recall  various notions of twisted modules for a \voa following \cite{DLM3} and  discuss some  important  concepts such as  $g$-rationality, regularity, and $C_2$-cofiniteness from \cite{Z} and \cite{DLM2}, \cite{DLM3} .  

 We first define the  $C_2$-cofiniteness \cite{Z}.
\begin{de}
We say that a \voa $V$ is $C_2$-cofinite if $V/C_2(V)$ is finite dimensional, where $C_2(V)=\langle v_{-2}u|v,u\in V\rangle.$
\end{de}

Fix   vertex operator algebra $V$ and an automorphism $g$ of $V$ of finite order $T$.  Decompose  $V$ into a direct sum of eigenspaces of $g:$
\begin{equation*}\label{g2.1}
V=\bigoplus_{r\in \Z/T\Z}V^r,
\end{equation*}
where $V^r=\{v\in V|gv=e^{-2\pi ir/T}v\}$.
We use $r$ to denote both
an integer between $0$ and $T-1$ and its residue class modulo $T$ in this
situation.

\begin{de} \label{weak}
A {\em weak $g$-twisted $V$-module} $M$ is a vector space equipped
with a linear map
\begin{equation*}
\begin{split}
Y_M: V&\to (\End\,M)[[z^{1/T},z^{-1/T}]]\\
v&\mapsto\displaystyle{ Y_M(v,z)=\sum_{n\in\frac{1}{T}\Z}v_nz^{-n-1}\ \ \ (v_n\in
\End\,M)}
\end{split}
\end{equation*}
 satisfying  the following:  for all $0\leq r\leq T-1,$ $u\in V^r$, $v\in V,$
$w\in M$,
\begin{eqnarray*}
& &Y_M(u,z)=\sum_{n\in \frac{r}{T}+\Z}u_nz^{-n-1} \label{1/2},\\
& &u_lw=0~~~
\mbox{for}~~~ l\gg 0,\label{vlw0}\\
& &Y_M({\mathbf 1},z)=Id_M,\label{vacuum}
\end{eqnarray*}
 \begin{equation*}\label{jacobi}
\begin{array}{c}
\displaystyle{z^{-1}_0\delta\left(\frac{z_1-z_2}{z_0}\right)
Y_M(u,z_1)Y_M(v,z_2)-z^{-1}_0\delta\left(\frac{z_2-z_1}{-z_0}\right)
Y_M(v,z_2)Y_M(u,z_1)}\\
\displaystyle{=z_2^{-1}\left(\frac{z_1-z_0}{z_2}\right)^{-r/T}
\delta\left(\frac{z_1-z_0}{z_2}\right)
Y_M(Y(u,z_0)v,z_2)},
\end{array}
\end{equation*}
where $\delta(z)=\sum_{n\in\Z}z^n$ and
all binomial expressions  are to be expanded in nonnegative
integral powers of the second variable.
\end{de}

\begin{de}\label{ordinary}
A $g$-{\em twisted $V$-module} is
a $\C$-graded weak $g$-twisted $V$-module $M:$
\begin{equation*}
M=\bigoplus_{\lambda \in{\C}}M_{\lambda}
\end{equation*}
where $M_{\l}=\{w\in M|L(0)w=\l w\}$ and $L(0)$ is the component operator of $Y(\omega,z)=\sum_{n\in \Z}L(n)z^{-n-2}.$ We also require that
$\dim M_{\l}$ is finite and for fixed $\l,$ $M_{\frac{n}{T}+\l}=0$
for all small enough integers $n.$ If $w\in M_{\l}$ we call $\l$  the {\em weight} of
$w$ and write $\l=\wt w.$
\end{de}

\begin{de}\label{admissible}  Let $\Z_+$  be the set of nonnegative integers.
 An {\em admissible} $g$-twisted $V$-module
is a  $\frac1T{\Z}_{+}$-graded weak $g$-twisted $V$-module $M:$
\begin{equation*}
M=\bigoplus_{n\in\frac{1}{T}\Z_+}M(n)
\end{equation*}
satisfying
\begin{equation*}
v_mM(n)\subseteq M(n+\wt v-m-1)
\end{equation*}
for homogeneous $v\in V,$ $m,n\in \frac{1}{T}{\Z}.$
\ed

In the case  $g$ is the identity map,  we have the notions of  weak, ordinary and admissible $V$-modules \cite{DLM3}.

We also need the notion of  the contragredient module. For  an admissible $g$-twisted $V$-module $M=\bigoplus_{n\in \frac{1}{T}\Z_+}M(n),$  the contragredient module $M'$
is defined as
\begin{equation*}
M'=\bigoplus_{n\in \frac{1}{T}\Z_+}M(n)^{*},
\end{equation*}
where $M(n)^*=\Hom_{\C}(M(n),\C).$ The vertex operator
$Y_{M'}(a,z)$ is defined for $a\in V$ via
\begin{eqnarray*}
\langle Y_{M'}(a,z)f,u\rangle= \langle f,Y_M(e^{zL(1)}(-z^{-2})^{L(0)}a,z^{-1})u\rangle,
\end{eqnarray*}
where $\langle f,w\rangle=f(w)$ is the natural paring $M'\times M\to \C.$
Then  $(M',Y_{M'})$ is an admissible $g^{-1}$-twisted $V$-module \cite{FHL}, \cite{X}. 
The contragredient module $M'$ for a $g$-twisted $V$-module $M$ can also be defined in the same fashion.  In this case,
$M'$ is a $g^{-1}$-twisted $V$-module. Moreover, $M$ is irreducible if and only if $M'$ is irreducible.
$V$ is called selfdual if $V$ and $V'$ are isomorphic $V$-modules.

\begin{de}
A \voa $V$ is called $g$-rational, if the  admissible $g$-twisted module category is semisimple. $V$ is called rational if $V$ is $1$-rational where $1$ is the identity map on $V.$
\end{de}

The following results from \cite{DLM3}  tell us why rationality is important.
\begin{thm}\label{grational}
Assume that $V$ is $g$-rational.  Then

(1) Any irreducible admissible $g$-twisted $V$-module $M$ is a $g$-twisted $V$-module. Moreover, there exists a number $\l \in \mathbb{C}$ such that  $M=\oplus_{n\in \frac{1}{T}\mathbb{Z_+}}M_{\l +n}$ where $M_{\lambda}\neq 0.$ The $\l$ is called the conformal weight of $M.$

(2) There are only finitely many inequivalent  irreducible admissible  $g$-twisted $V$-modules.

(3) If $V$ is also $C_2$-cofinite and $g^i$-rational for all $i\geq 0$ then the central charge $c$ and the conformal weight $\l$ of any irreducible $g$-twisted $V$-module $M$ are rational numbers.
\end{thm}

A \voa $V$ is called regular if every weak $V$-module is a direct sum of irreducible $V$-modules \cite{DLM2}.
A \voa $V=\oplus_{n\in \Z}V_n$  is said to be of CFT type if $V_n=0$ for negative
$n$ and $V_0=\C {\bf 1}.$
It is proved in \cite{ABD}, \cite{L}  that if  $V$ is of CFT type, then regularity is equivalent to rationality and $C_2$-cofiniteness. Also $V$ is regular if and only if the weak module category is semisimple \cite{DYu}.

The following results from \cite{M} and \cite{CM} play important roles in this paper.
\begin{thm}\label{M} If $V$ is a regular vertex operator algebra of CFT type and $G$ is a finite  solvable subgroup of $\Aut(V)$ then
$V^G$ is regular.
\end{thm} 

Using a result from \cite{ADJR} and Theorem \ref{M} we have 
\begin{thm}\label {ADJR} If $V$ is regular vertex operator algebra of CFT type then $V$ is $g$-rational for any finite order automorphism 
$g$ of $V.$ 
\end{thm} 

\section{Modular invariance}

We  present  the modular invariance property of the trace functions in orbifold theory  from \cite{Z} and \cite{DLM4} in this section.

In order to discuss the modular invariance of trace functions we consider  the action of $\Aut(V)$ on twisted modules \cite{DLM4}. Let $g, h$ be two automorphisms of $V$ with $g$ of finite order. If $(M, Y_M)$ is a weak $g$-twisted $V$-module, then  $(M\circ h, Y_{M\circ h})$ is   a weak $h^{-1}gh$-twisted  $V$-module where $M\circ h\cong M$ as vector spaces and
\begin{equation*}
Y_{M\circ h}(v,z)=Y_M(hv,z)
\end{equation*}
for $v\in V.$
This defines a right action of $\Aut(V)$ on weak twisted $V$-modules and on isomorphism
classes of weak twisted $V$-modules. For short, we write
\begin{equation*}
(M,Y_M)\circ h=(M\circ h,Y_{M\circ h})= M\circ h.
\end{equation*}

Assume that $g,h$ commute.  Denote by $\mathscr{M}(g)$ the equivalence classes of irreducible $g$-twisted $V$-modules and set $\mathscr{M}(g,h)=\{M \in \mathscr{M}(g)| h\circ M\cong M\}.$  Then $\mathscr{M}(g,h)$ is a subset of $\mathscr{M}(g).$
By  Theorems \ref{grational} and  \ref{ADJR} ,   if $V$ is a regular vertex operator algebra of CFT type,  both $\mathscr{M}(g)$ and $\mathscr{M}(g,h)$ are finite sets.  Let  $M\in \mathscr{M}(g,h).$ There is a linear isomorphism $\varphi(h)$ from $M$ to $M$ such that  :
\begin{equation*}
\varphi(h)Y_M(v,z)\varphi(h)^{-1}=Y_M(hv,z)
\end{equation*}
for $v\in V.$  Note that  $\varphi(h)$ is unique up to a nonzero scalar. If $h=1$ we simply take $\varphi(1)=1.$
For $v\in V$ we set
\begin{equation*}
Z_M(v, (g,h),\tau)=\tr_{_M}o(v)\varphi(h) q^{L(0)-c/24}=q^{\lambda-c/24}\sum_{n\in\frac{1}{T}\Z_+}\tr_{_{M_{\l+n}}}o(v)\varphi(h)q^{n}
\end{equation*}
which is a holomorphic function on the upper half plane $\H$ \cite{DLM4} with $q=e^{2\pi i\tau}.$  Note that $Z_M(v, (g,h),\tau)$ is defined up to a nonzero scalar. 
If   $h=1$ and $v=\1$  then $Z_M(\1,(g,1),\tau)$ is called the $q$-character of $M$ and is denoted by $\chi_M(\tau).$ 

For the modular invariance we need another vertex operator algebra $(V, Y[~], \1, \tilde{\omega})$ associated to $V$
in \cite{Z}.  Here $\tilde{\omega}=\omega-c/24$ and
$$Y[v,z]=Y(v,e^z-1)e^{z\cdot \wt v}=\sum_{n\in \Z}v[n]z^{-n-1}$$
for homogeneous $v.$ We also write
$$Y[\tilde{\omega},z]=\sum_{n\in \Z}L[n]z^{-n-2}.$$
If $v\in V$ is homogeneous in the second vertex operator algebra, we denote its weight by $\wt [v].$

From now on we assume that $V$ is a regular vertex operator algebra of CFT type and $G$ a finite automorphism group of $V.$ 
Set $\M=\cup_{g\in G}\M(g).$
As usual, let $P(G)$ denote the ordered commuting pairs in $G.$
For $(g,h)\in P(G)$ and $M\in \mathscr{M}(g,h),$ $Z_M(v,(g,h),\tau)$ is a function
on $V\times \H.$ Let ${\cal W}$ be the vector space spanned by these functions. It is clear that
the dimensional of $\W$ is equal to $\sum_{(g,h)\in P(G)}|\mathscr{M}(g,h)|.$
Define an action of the modular group $\Gamma=SL(2,\Z)$ on $\W$ such that
\begin{equation*}
Z_M|_\gamma(v,(g,h),\tau)=(c\tau+d)^{-{\rm wt}[v]}Z_M(v,(g,h),\gamma \tau),
\end{equation*}
where
\begin{equation}\label{e1.11}
\gamma: \tau\mapsto\frac{ a\tau + b}{c\tau+d},\ \ \ \gamma=\left(\begin{array}{cc}a & b\\ c & d\end{array}\right)\in\Gamma.
\end{equation}
 We let $\gamma\in \Gamma$ act on the right of $P(G)$ via
$$(g, h)\gamma = (g^ah^c, g^bh^d ).$$

Using Theorem \ref{ADJR}，  we have the following results \cite{DLM4} .
\begin{thm}\label{minvariance} (1) There is a representation $\rho: \Gamma\to GL(\W)$
such that for $(g, h)\in P(G),$   $\gamma =\left(\begin{array}{cc}a & b\\ c & d\end{array}\right)\in \Gamma,$
and $M\in \mathscr{M}(g,h),$
$$
Z_{M}|_{\gamma}(v,(g,h),\tau)=\sum_{N\in \mathscr{M}(g^ah^c,{g^bh^d)}} \gamma_{M,N} Z_{N}(v,(g, h)\gamma, ~\tau)$$
where $\rho(\gamma)=(\gamma_{(M,g,h), (N, g_1,h_1)})$ for $N\in\M$ and $(g_1,h_1)\in P(G)$ and we use  $ \gamma_{M,N}$
to denote $ \gamma_{(M,g,h), (N, g^ah^c,g^bh^d)}$ for $N\in  \mathscr{M}(g^ah^c,g^bh^d)$ for short.
That is,
$$Z_{M}(v,(g,h),\gamma\tau)=(c\tau+d)^{{\rm wt}[v]}\sum_{N\in \mathscr{M}(g^ah^c,{g^bh^d)}} \gamma_{M,N} Z_{N}(v,(g, h)\gamma, ~\tau).$$

(2) The cardinalities $|\mathscr{M}(g,h)|$ and $|\mathscr{M}(g^ah^c,g^bh^d)|$ are equal for any $(g,h)\in P(G)$ and
$\gamma\in \Gamma.$ In particular, the number of inequivalent irreducible $g$-twisted $V$-modules is exactly the number of irreducible $V$-modules which are  $g$-stable.
\end{thm}

We remark that if $G=\{1\},$ Part (1) was established in \cite{Z}.  Part (2) in this case is trivial.

\section{Congruence subgroup property}

We are now in a position to state the main result in this paper.
\begin{thm}\label{main} Let $V$ be a regular, selfdual  vertex operator algebra of CFT type and $G$ a finite automorphism group of $V.$ 
Then the kernel of $\rho$ is a congruence subgroup.
\end{thm}

The proof of Theorem \ref{main} is based on the congruence subgroup property for $G=\{1\}$ given in \cite{DLN}. In this case
we denote $Z_M(v,(1,1),\tau)$ simply by $Z_M(v,\tau).$ Let $\{M^0,...,M^p\}$ be the inequivalent  irreducible $V$-modules.  Denote  the weight of $M^i$ by $\lambda_i.$ 
By Theorem \ref{grational}, $\lambda_i$ and $c$ are rational numbers.  Let $s$ be the smallest positive integer such that $s(\lambda_i-c/24)$ is an integer for all $i.$  The following result was obtained in \cite{DLN}.
\begin{thm}\label{DLN} If $V$ is a regular and selfdual and $G=\{1\},$  the kernel of $\rho$ contains the congruence subgroup $\Gamma(s).$
\end{thm}

Note that we do not need assumption that $V$ is of CFT type in Theorem \ref{DLN}. This assumption is necessary in the current situation as Theorem \ref{M} requires it. 

\begin{lem}\label{lem}  Let $V$ and $G$ be as in Theorem \ref{main}.  Suppose that  $(g,h)\in P(G).$ Then there is a positive integer $m_{g,h}$ such that for $M\in \M(g,h)$ and $v\in V_{[n]}$ 
$Z_M(v,(g,h),\tau)$ is a modular form on the congruence subgroup $\Gamma(m_{g,h})$ of weight $n.$ 
\end{lem}

\begin{proof}  Let $H$ be the subgroup of $G$ generated by $g,h.$ Then $H$ is a finite abelian group. By Theorem \ref{M}, $V^H$
is a regular vertex operator algebra.  We expect to apply Theorem \ref{DLN} to vertex operator algebra $V^H.$  So we need to verify that $V^H$ is selfdual.

We know from \cite{DM} and \cite{DLM1} that $V$ has the following decomposition
$$V=\oplus_{\mu\in \hat H}V^{\mu}$$
where $\hat H$ is the set of irreducible characters of $H$ and $V^{\mu}$ is the subspace of $V$ corresponding to the character $\mu.$  Since $V$ is selfdual, it is easy to see that the dual of $V^{\mu}$ is $V^{\mu^{-1}}.$ In particular, if $\mu=1$ is the trivial character of $H,$ $V^H=V^1$ is selfdual.

From Theorem \ref{DLN} there is a positive integer $m_{g,h}$ such that $Z_W(v,\tau)$ is a modular form of weight 
$n$ on congruence subgroup $\Gamma(m_{g,h})$  for any irreducible $V^H$-module $W$ and any $v\in V^H_{[n]}.$ The main idea is to express $Z_M(v,(g,h),\tau)$ 
as a linear combination of $Z_W(v,\tau)$ for some  irreducible $V^H$-modules $W.$  

Since $V^H$ is rational, it follows from a result in \cite{DRX} that any irreducible $V^H$-module occurs in an irreducible $k$-twisted $V$-module for  some $k\in H.$ Let $H_M=\{k\in H|M\circ k\cong M\}$ be the stabilizer of $M$ in $H.$ Then $H_M$ acts on $M$ projectively. That is, there exists a 2-cocyle $\alpha\in H^2(H_M,S^1)$ such that $\varphi(k_1)\varphi(k_2)=\alpha(k_1,k_2) \varphi(k_1k_2)$ for $k_i\in H_M$ where $S^1$ is the group of unit circle. Then the twisted group algebra $\C^{\alpha}[H_M]$ 
is a finite dimensional semisimple associative algebra. Denote by $\Lambda_M$ the set of irreducible characters of $\C^{\alpha}[H_M].$ For each $\lambda\in \Lambda_M$ let $W_{\lambda}$ be the corresponding irreducible $\C^{\alpha}[H_M]$-module. 
This leads to the following decomposition
$$M=\oplus_{\lambda\in  \Lambda_M}W_{\lambda}\otimes M_{\lambda}$$
where $M_{\lambda}$ is the multiplicity space of $W_{\lambda}$ in $M.$ From \cite{DM}, \cite{DLM1}, \cite{DY}, \cite{MT}, 
$M_{\lambda}$ is an irreducible $V^H$-module. In particular, $M$ is a completely reducible $V^H$-module.

We can assume that $ \varphi(h)$ has finite order $K.$ Then $M=\oplus_{j=0}^{K-1}M_j$ where $M_j$ is the eigenspace 
of  $ \varphi(h)$  with eigenvalue $e^{2\pi i j/K}.$ Using the identity 
$$\varphi(h)Y_M(v,z)\varphi(h)^{-1}=Y_M(hv,z)$$
we see that $\varphi(h)$ and $Y_M(v,z)$ commute for $v\in V^H.$ This implies that each $M_j$ is a $V^H$-submodule 
of $M$ and is a direct sum of irreducible $V^H$-modules $M_{\lambda}$ for $\lambda\in \Lambda_M.$ It follows from the definition
that 
$$Z_M(v,(g,h),\tau)=\sum_{j=0}^{K-1}e^{2\pi ij/K}\tr_{M_j}o(v)q^{L(0)-c/24}.$$
Clearly, each  $Z_{M_j}(v,\tau)=\tr_{M_j}o(v)q^{L(0)-c/24}$ is a modular form on $\Gamma(m_{g,h})$ of weight $n$ for any 
$v\in V^H_{[n]}.$ So is $Z_M(v,(g,h),\tau).$
\end{proof}

{\bf Proof of  Theorem \ref{main}:} Let $m$ be the least common multiple of all $m_{g,h}$ for $(g,h)\in P(G).$ Then for any $v\in V^G_{[n]}$ and $M\in \M,$ $ Z_M(v,(g,h),\tau)$ is a modular form on $\Gamma(m)$ of weight $n.$ That is, the kernel of  $\rho$ contains the congruence subgroup $\Gamma(m).$  The proof is complete.
\vspace{0.5cm}

Although we do not know $V^G$ is rational or $C_2$-cofinite when $G$ is not abelian, but we still have:
\begin{coro}\label{coro} If $V$ is a regular, selfdual vertex operator algebra of CFT type and $G$ is any finite automorphism group of $V$ then
for any irreducible $V^G$-module $N$ occurring in an  irreducible $g$-twisted $V$-module $M$ with $g\in G,$ 
$Z_{N}(v,\tau)$ is a modular form of weight $n$ on  the congruence subgroup $\Gamma(m)$ where $m$ 
is the same as before and $v\in V^G_{[n]}.$ In particular,  the $q$-character $\chi_{N}(\tau)$ is a modular function on
$\Gamma(m).$ 
\end{coro}
\begin{proof}  First, we recall  the irreducible $V^G$-module occurring in an irreducible $g$-twisted $V$-module $M$
from \cite{MT} and \cite{DRX}.  The above discussion tells us  that $h\mapsto \varphi(h)$ gives a projective representation of $G_M$ on $M.$ Let $\alpha_M\in H^2(G_M,S^1)$ be the corresponding 2-cocycle. Then 
$$\varphi(h)\varphi(k)=\alpha_M(h,k)\varphi(hk),$$
$$\varphi(h)Y_M(v,z)\varphi(h)^{-1}=Y_M(hv,z)$$
for $h,k\in G_M$ and $v\in V.$ As before,  $M$ is a module for the twisted group algebra $\C^{\alpha_M}[G_M]$
and we have decomposition 
$$M=\oplus_{\lambda\in \Lambda_{G_M,\a_M}}W_{\lambda}\otimes M_{\lambda}$$
where $ \Lambda_{G_M,\a_M}$ is the set of irreducible characters of  $\C^{\alpha_M}[G_M],$ 
$W_{\lambda}$ is the irreducible $\C^{\alpha_M}[G_M]$-module  with character $\lambda$ and $ M_{\lambda}$ is the multiplicity of $W_{\lambda}$ in $M.$ 

It follows from \cite{MT} and \cite{DRX} that each $ M_{\lambda}$ is an irreducible $V^G$-module. Moreover,
$M_{\lambda}$ and $M_{\mu}$ are not isomorphic if $\lambda, \mu$ are different. 
Let $\lambda, \mu$ be two irreducible characters of $\C^{\alpha_M}[G_M].$ The following orthogonal relation
 holds:
 $$\frac{1}{|G_M|}\sum_{h\in G_M}\lambda(h)\overline{\mu(h)}=\delta_{\lambda,\mu}$$
where $\overline{\mu(h)}$ is the complex conjugate of $ \mu(h)$ by Lemma 4.3 of \cite{DRX}.
We can now assume that $N=M_{\lambda}$ for some $\lambda\in \Lambda_{G_M,\alpha_M}.$  Then we can easily get
$$Z_{M_{\lambda}}(v,\tau)=\frac{1}{|G_M|}\sum_{h\in G_M}\overline{\lambda(h)}Z_M(v,(g,h),\tau).$$
By Theorem \ref{main}, each $Z_M(v,(g,h),\tau)$ is a modular form on $\Gamma(m).$ The result follows immediately.
\end{proof}

From Corollary \ref{coro} we know that $q$-character $\chi_{V^G}(\tau)$ of $V^G$ is a modular function. This suggests 
 the following conjecture:
\begin{conj}  Let $U$ be a simple  vertex operator algebra. Then 

(1) $U$ is rational if and ony if the $q$-character of each irreducible $U$-module is a modular function on a same  congruence subgroup.

(2) $U$ is rational if and only if the $q$-character of $U$  is a modular function on a congruence subgroup.
\end{conj} 

Clearly,  Conjecture (2) impies (1).  These conjectures  tell us  that the modularity of the $q$-character  of a  vertex operator algebra  is unique for rational vertex operator algebras. If this conjecture is true, it would imply that $V^G$ is rational for any finite automorphism group $G$ of $V.$  

The following conjecture is a converse of  Zhu's result \cite{Z}.
\begin{conj}
  Let $U$ be a simple, $C_2$-cofinite  vertex operator algebra. Let  $M^0,...,M^p$ be the irreducible $U$-modules. Then $U$ is rational if and only if  the space spanned by $\{Z_{M^i}(v,\tau)|i=0,...,p\}$ affords a representation  of the modular group $\Gamma.$
\end{conj}

\section{Holomorphic orbifolds}

Recall that vertex operator algebra $V$ is called holomorphic if it is rational and has a unique irreducible module,  namely $V$ itself.   Let $V$ be a $C_2$-cofinite, holomorphic vertex operator algebra of CFT type and $G$ a finite automorphism group of $V.$
Then for each $g\in G$ there is a unique irreducible $g$-twisted $V$-module $V(g)$ by Theorem \ref{minvariance}. For short
we write $Z(v,(g,h),\tau)$ for  $Z_{V(g)}(v,(g,h),\tau)$ for $(g,h)\in P(G)$  and $Z(g,h,\tau)$ for $Z(\1,(g,h),\tau).$

\begin{prop} The following hold:

(1) For any $k\in G,$ $Z(v,(kgk^{-1},khk^{-1}),\tau)=Z(v,(g,h),\tau).$

(2) For any $\gamma=\left(\begin{array}{cc}a & b\\ c & d\end{array}\right)\in \Gamma,$  and $(g,h)\in P(G)$ and $v\in V^G_{[n]}$
$$Z(v,(g,h),\gamma\tau)=(c\tau+d)^{{\rm wt}[v]} \gamma_{(g,h),(g^ah^c,g^bh^d)} Z(v,(g^ah^c, g^bh^d), \tau).$$

(3)  $Z(v,(g,h),\tau)$ is a modular form of weight $n$ on the congruence subgroup $\Gamma(m)$ where  $(g,h)\in P(G),$
$v\in V^G_{[n]},$ and  $m$ is given in Theorem \ref{main}. 

(4) The coefficient of each power of $q$ in  $Z(g,h),\tau)$ defines a projective character of $C_G(g).$ 

In particular, (1)-(4) hold for $Z(g,h,\tau).$ 
\end{prop}

\begin{proof}  Part (3)  is a special case of Theorem \ref{main}. The rest was given in \cite{DLM4}. 
\end{proof}

In the case that $V$ is holomorphic and $h$ is a power of $g,$ the modularity of $Z(g,h,\tau)$ was obtained previously
in \cite{EMS}.

We know from \cite{D}, \cite{DGH} that  moonshine vertex operator algebra $V^{\natural}$ \cite{FLM}  is a $C_2$-cofinite,
holomorphic vertex operator algebra of CFT type.  Moreover, the automorphism group of $V^{\natural}$ is the monster
simple group \cite{FLM}.  In this case,  the functions $Z(g,h,\tau)$ are the candidates in the generalized moonshine conjecture \cite{N}, \cite{DLM4}. 
Parts (1), (2), (4)  were proved to hold in \cite{DLM4}. What new here is that each generalized McKay-Thompson series
$Z(g,h, \tau)$ is a modular function over a congruence subgroup $\Gamma(m).$  So unfinished business in proving the generalized moonshine conjecture is to extend the group $\Gamma(m)$ to a genus zero subgroup of $SL(2,\R).$


\begin{thebibliography}{FLMM}

\bibitem[ABD]{ABD}
T. Abe, G. Buhl and C. Dong, Rationality, regularity, and $C_2$-cofiniteness,
{\em Trans. Amer. Math. Soc.} {\bf 356} (2004), 3391-3402.

\bibitem[ADJR]{ADJR}C. Ai, C. Dong, X. Jiao and L. Ren, The irreducible modules and fusion rules for the parafermion vertex operator algebras,  arXiv:1412.8154.

\bibitem[B]{B} R.  Borcherds,  Monstrous moonshine and monstrous Lie superalgebras, {\em Invent. Math.} {\bf 109} (1992), 405–444.

\bibitem[C]{C}
J.~Cardy, \emph{Operator content of two-dimensional conformally invariant
  theories}, {\em Nuclear Phys. B} \textbf{270} (1986), no.~2, 186--204.

\bibitem[CM]{CM} S. Carnahan and M. Miyamoto, Regularity of fixed-point vertex operator subalgebras, arXiv:1603.05645.


\bibitem[CN]{CN}  J.  Conway  and S. Norton, Monstrous moonshine, {\em  Bull. London. Math. Soc.}  {\bf 12} (1979), 308–339 .


\bibitem[D]{D}  C. Dong, Representations of the moonshine module vertex operator algebra, {\em Contemporary Math.}
{\bf 175} (1994), 27-36.

\bibitem[DGH]{DGH} C. Dong, R. Griess Jr. and G. Hoehn,
Framed vertex operator algebras, codes and the moonshine module,
{\em Comm. Math. Phys.} {\bf 193} (1998), 407-448.

\bibitem[DLM1]{DLM1} C. Dong, H. Li and G. Mason,
Compact automorphism groups of \voas , {\em Int. Math. Res. Not.} {\bf 18}  (1996), 913-921.

\bibitem[DLM2]{DLM2} C. Dong, H. Li and G. Mason,
Regularity of rational vertex operator algebras, {\em  Adv. Math.} {\bf 132} (1997), 148-166.

\bibitem[DLM3]{DLM3} C. Dong, H. Li and G. Mason,
Twisted representations of vertex operator algebras, {\em Math. Ann.}
{\bf  310} (1998), 571--600.


\bibitem[DLM4]{DLM4} C. Dong, H. Li and G. Mason,
Modular-invariance of trace functions in orbifold theory and generalized moonshine, {\em  Comm. Math. Phys.}
{\bf 214} (2000), 1-56.

\bibitem[DLN]{DLN} C. Dong, X. Lin, S. Ng, Congruence property in conformal field theory, {\em Algebra \& Number Theory,} {\bf 9} (2015), 2121-2166.

\bibitem[DM]{DM} C. Dong and G. Mason, On quantum Galois theory, {\em Duke Math. J.} {\bf 86} (1997), 305-321.

\bibitem[DRX]{DRX}  C. Dong, L. Ren, F. Xu, On Orbifold Theory, arXiv:1507.03306.
\bibitem[DY]{DY} C. Dong and G. Yamskulna,
Vertex operator algebras, Generalized double and dual pairs,
{\em Math. Z.} {\bf 241} (2002), 397-423.

\bibitem[DYu]{DYu}C. Dong and N. Yu, $Z$-graded weak modules and regularity, {\em Comm. Math. Phys.} {\bf 316} (2012), 269-277.

\bibitem[EMS]{EMS} J. Ekeren,  S.  M\"{o}ller and  N. Scheithauer, Construction and classification of
holomorphic vertex operator algebras,  arXiv:1507.08142 


\bibitem[FHL]{FHL}
I. Frenkel, Y. Huang and J. Lepowsky, On axiomatic approaches to vertex operator algebras and modules, {\em Mem. Amer. Math. Soc.} {\bf 104} 1993.



\bibitem[FLM]{FLM} I. Frenkel, J. Lepowsky and A. Meurman,
Vertex Operator Algebras and the Monster, {\em Pure and Applied
Math.,} {\bf Vol. 134,} Academic Press, Boston, 1988.

\bibitem[G]{G} R. Griess,  The friendly giant, {\em  Invent. Math.} {\bf  69} (1982), 1–102.

 \bibitem[KP]{KP} V. Kac and D. Peterson, Infinite dimensional lie algebras, theta functions and modular forms,  {\em Adv. Math.} {\bf53} (1984), 125-264.

\bibitem[L]{L} H. Li, Some finiteness properties of regular vertex operator algebras,  {\em J. Algebra}  {\bf 212}  (1999),  495-514. 


 \bibitem[M]{M} M. Miyamoto, $ C_2$-cofiniteness of cyclic-orbifold models, {\em  Comm. Math. Phys.} {\bf  335}  (2015), 1279锟紺1286.

\bibitem[MT]{MT}M. Miyamoto and K. Tanabe, Uniform product of $A_{g,n}(V)$ for an orbifold
model V and $G$-twisted Zhu algebra, {\em J. Algebra} {\bf 274} (2004), 80-96.

\bibitem[N]{N} S. Norton,  Generalized moonshine., {\em Proc. Symp. Pure. Math., American Math. Soc.}  {\bf  47} (1987), 208–209.
\bibitem[R]{R} A. Rocha-Caridi, In: Vertex operators in mathematics and physics, Edited by J. Lepowsky, S. Mandelstam and I. M. Singer. MSRI Publications, 3, Springer-Verlag, New York, 1985.





\bibitem[X]{X}
X. Xu, Introduction to Vertex Operator Superalgebras and Their Modules,
{\em Mathematics and its Applications,} {\bf Vol. 456},
 Kluwer Academic Publishers, Dordrecht, 1998.



\bibitem[Z]{Z}
Y. Zhu, Modular invariance of characters of vertex operator algebras,
{\em J. Amer, Math. Soc.}  {\bf 9} (1996), 237-302.
\end{thebibliography}
\end{document}